\documentclass{amsart}
\usepackage{amsaddr}
\usepackage{tikz,pgflibraryshapes}
\usetikzlibrary{calc}
\usetikzlibrary{intersections}
\usetikzlibrary{arrows,shapes,positioning}
\usetikzlibrary{decorations.markings}

\usepackage[pdftex,pdfcreator={LaTeX},pdfproducer={PDFLaTeX},colorlinks,
    bookmarks,
    pdftitle={},
    pdfauthor={},
    pdfkeywords={},
    linkcolor=blue,                  
    citecolor=red,                   
    filecolor=red,                   
    urlcolor=cyan                    
]{hyperref}

\newcommand{\texfilename}{GirthRegular_arXiv}
\newcommand{\Title}[2][]{\title{#2}\thanks{#1}}
\newcommand{\Author}[1][]{}

\newcommand{\auth}[3][]{\author[#2]{#3}}
\newcommand{\ead}[1]{\email{#1}}
\newcommand{\naddress}[2][]{\address[#1]{#2}}
\newcommand{\raddress}[2][]{\address[#1]{#2}}
\newcommand{\scl}[2][2000]{\subjclass[#1]{#2}}
\newcommand{\printscl}{}
\newcommand{\sep}{, }
\newcommand{\Maketitle}{\maketitle}

\newenvironment{frmatter}{}{}
\usepackage[english]{babel}
\usepackage[utf8x]{inputenc}
\usepackage{array}
\usepackage{graphicx}
\usepackage{enumerate}
\usepackage{color}
\usepackage[small]{caption}
\usepackage[normalem]{ulem}
\usepackage{latexsym,amssymb,amsmath,amsopn,amsthm,amscd,amsfonts,url}

\pgfrealjobname{\texfilename}

\newcolumntype{C}[1]{>{\centering\arraybackslash}m{#1\linewidth}}

\makeatletter
\newcommand{\pitem}{%
\item[(\arabic{enumi}')]\protected@edef\@currentlabel{\arabic{enumi}'}%
}
\makeatother

\newcommand{\doi}[1]{\href{http://dx.doi.org/#1}{\texttt{doi:#1}}}

\newcommand{\G}{\ensuremath{\Gamma}}
\newcommand{\D}{\ensuremath{\mathbb{D}}}
\newcommand{\Z}{\ensuremath{\mathbb{Z}}}

\newcommand{\Aut}{\ensuremath{\operatorname{Aut}}}

\newcommand{\out}{\ensuremath{\operatorname{out}}}

\newcommand{\X}{\ensuremath{\mathcal{X}}}
\newcommand{\Y}{\ensuremath{\mathcal{Y}}}

\newcommand{\C}{\ensuremath{\mathcal{C}}}

\newcommand{\mD}{\ensuremath{\mathcal{D}}}
\newcommand{\Cay}{\ensuremath{\operatorname{Cay}}}
\newcommand{\Tr}{\ensuremath{\operatorname{Tr}}}

\newcommand{\M}{\ensuremath{\mathcal{M}}}
\newcommand{\T}{\ensuremath{\mathcal{T}}}
\newcommand{\s}{\ensuremath{\mathcal{S}}}

\newcommand{\lr}{\leftrightarrow}

\newcommand{\thmcounter}{definition}
\newtheorem{\thmcounter}{Definition}[section]
\newtheorem{theorem}[\thmcounter]{Theorem}
\newtheorem{proposition}[\thmcounter]{Proposition}
\newtheorem{lemma}[\thmcounter]{Lemma}
\newtheorem{corollary}[\thmcounter]{Corollary}

\begin{document}
\newtheorem{question}[\thmcounter]{Question}
\newtheorem{note}[\thmcounter]{Note}

\begin{frmatter}

\Title[Supported in part by the Slovenian Research Agency,
       projects J1-5433, J1-6720, and P1-0294.]{Girth-regular graphs}

\auth[fmf,imfm]{P.~Poto\v{c}nik}{Primo\v{z} Poto\v{c}nik}
\ead{primoz.potocnik@fmf.uni-lj.si}

\address[fmf]{Faculty of Mathematics and Physics, University of Ljubljana, \\
              Jadranska 19, SI-1000 Ljubljana, Slovenia}

\naddress[imfm]{Institute of Mathematics, Physics and Mechanics, \\
                Jadranska 19, SI-1000 Ljubljana, Slovenia}

\Author

\auth[fmf]{J.~Vidali}{Jano\v{s} Vidali}
\ead{janos.vidali@fmf.uni-lj.si}

\raddress[fmf]{Faculty of Mathematics and Physics, University of Ljubljana, \\
               Jadranska 19, SI-1000 Ljubljana, Slovenia}

\Author

\scl[2000]{05C38}


\begin{abstract}
We introduce a notion of a {\em girth-regular} graph
as a $k$-regular graph for which there exists a non-descending sequence
$(a_1, a_2, \dots, a_k)$ (called the {\em signature}) giving,
for every vertex $u$ of the graph,
the number of girth cycles the edges with end-vertex $u$ lie on.
Girth-regularity generalises two very different aspects of symmetry
in graph theory:
that of vertex transitivity and that of distance-regularity.
For general girth-regular graphs,
we give some results on the extremal cases of signatures.
We then focus on the cubic case and provide a characterisation
of cubic girth-regular graphs of girth up to $5$.

\printscl
\end{abstract}

\end{frmatter}

\Maketitle


\section{Introduction}
\label{sec:intro}

This paper stems from our research
of finite connected vertex-transitive graphs of small girth.
The girth (the length of a shortest cycle in the graph)
is an important graph theoretical invariant that is often studied
in connection with the symmetry properties of graphs.
For example, cubic arc-transitive graphs
(a graph is called arc-transitive
if its automorphism group acts transitively on its {\em arcs},
where an arc is an ordered pair of adjacent vertices)
and cubic semisymmetric
(regular, edge-transitive but not vertex-transitive)
graphs of girth up to $9$ and $10$
have been studied in~\cite{cn07,km09} and~\cite{cz17}, respectively,
and tetravalent edge-transitive graphs of girths $3$ and $4$
have been considered in~\cite{pw07}.
Recently,
a classification of all cubic vertex-transitive graphs
of girth up to $5$ was obtained in~\cite{ejs19}.

In our investigation of vertex-transitive graphs of small girth, it became apparent to us that
 the condition of vertex-transitivity is almost never used in its full strength.
What was needed in most of the arguments was only a particular form of uniformity of the
distribution of girth cycles throughout the graph.
 Let us make this more precise.

For an edge $e$ of a graph $\Gamma$,
let $\epsilon(e)$ denote the number of girth cycles containing the edge $e$.
Let $v$ be a vertex of $\Gamma$
and let $\{e_1, \ldots, e_k\}$ be the set of edges incident to $v$
ordered in such a way that
$\epsilon(e_1) \le \epsilon(e_2) \le \ldots \le \epsilon(e_k)$.
Then the $k$-tuple $(\epsilon(e_1), \epsilon(e_2), \ldots, \epsilon(e_k))$
is called the {\em signature} of $v$.
A graph $\Gamma$ is called {\em girth-regular}
provided all of its vertices have the same signature.
The signature of a vertex is then called the signature of the graph.

We should like to point out
that girth-regular graphs of signature $(a,a,\ldots,a)$ for some $a$
have been introduced under the name {\em edge-girth-regular graphs}
in~\cite{jkm18},
where the authors focused on the families
of cubic and tetravalent edge-girth-regular graphs.

By definition, every girth-regular graph is regular
(in the sense that all its vertices have the same valence).
Further, it is clear that every vertex-transitive
as well as every semisymmetric graph is girth-regular.
Slightly less obvious is the fact that every distance-regular graph is also girth-regular.
The notion of girth-regularity is thus a natural generalisation of all these notions.
On the other hand,
examples of girth-regular graphs exist that are neither vertex-transitive,
nor semisymmetric nor distance-regular (for example, the truncation of a $3$-prism is such a graph; see Section~\ref{ssec:trunc}).

The central question we would like to propose and address in this paper
is the following:

\begin{question}
Given integers $k$ and $g$,
for which tuples $\sigma = (a_1, a_2, \ldots, a_k) \in \Z^k$
does a girth-regular graph of girth $g$ and signature $\sigma$ exist?
\end{question}

The above question seems to be very difficult if considered in its full generality.
We begin by
stating three theorems proved in Section~\ref{sec:bound}, which
give an upper bound on the entries $a_i$ of the signature
in terms of the valence $k$ and the girth $g$, and consider the case where this upper bound is attained.

\begin{theorem}
\label{the:1}
If $\Gamma$ is a girth-regular graph of valence $k$,
girth $g$, and signature $(a_1, \ldots, a_k)$,
then $a_k \le (k-1)^d$, where $d=\lfloor g/2 \rfloor$.
\end{theorem}

\begin{theorem}
\label{the:2}
If $\Gamma$ is a connected girth-regular graph of valence $k$,
girth $2d$ for some integer $d$,
and signature $(a_1, \ldots, a_k)$ such that $a_k = (k-1)^{d}$,
then $a_1 = a_2 = \ldots = a_k$
and $\Gamma$ is the incidence graph
of a generalised $d$-gon of order $(k-1,k-1)$.

In particular, if $k=3$, then $g \in \{4,6,8,12\}$ and $\Gamma$ is isomorphic to $K_{3,3}$ (if $g=4$), the Heawood graph (if $g=6$), the Tutte-Coxeter graph (if $g=8$) or to the Tutte 12-cage (if $g=12$).
\end{theorem}

For a description of the graphs mentioned in the above theorem,
see Note~\ref{note:prop:vertbound}.

\begin{theorem}
\label{the:3}
If $\Gamma$ is a connected $3$-valent girth-regular graph
of girth $2d+1$ for some integer $d$
and signature $(a_1,a_2,a_3)$ such that $a_3 = 2^d$,
then $\Gamma$ is isomorphic to $K_4$ or the Petersen graph.
\end{theorem}

In the second part of the paper,
we focus on $3$-valent graphs (also called {\em cubic} graphs)
and obtain a complete classification
of cubic girth-regular graphs of girth at most $5$
(see Theorem~\ref{the:main} below).
Prisms and M\"obius ladders are defined in Section~\ref{sec:g1234},
the notion of a dihedral scheme and truncation
is defined in Section~\ref{ssec:trunc},
and graphs arising from maps are discussed in Section~\ref{ssec:maps}.

\begin{theorem}
\label{the:main}
Let $\Gamma$ be a connected cubic girth-regular of girth $g$ with $g\le 5$.
Then either the signature of $\Gamma$ is $(0,1,1)$
and $\Gamma$ is a truncation of a dihedral scheme on some $g$-regular graph
(possibly with parallel edges),
or one of the following occurs:
\begin{enumerate}
\item $g=3$ and $\Gamma \cong K_4$ with signature $(2,2,2)$;
\item $g = 4$ and $\Gamma$ is isomorphic to a prism or to a Möbius ladder,
with signature $(4, 4, 4)$ if $\Gamma \cong K_{3,3}$,
signature $(2, 2, 2)$ if $\Gamma$ is isomorphic to the cube $Q_3$,
and signature $(1, 1, 2)$ otherwise;
\item $g=5$ and $\Gamma$ is isomorphic
to the Petersen graph with signature $(4,4,4)$,
or to the dodecahedron with signature $(2,2,2)$.
\end{enumerate}
\end{theorem}

Since every vertex-transitive graph is girth-regular,
the above result can be viewed as a partial generalisation
of the classification~\cite{cn07}
of arc-transitive cubic graphs of girth at most $9$
and also a recent classification~\cite{ejs19}
of vertex-transitive cubic graphs of girth at most $5$.

Unless explicitly stated otherwise, by a {\em graph}, we will always mean a finite simple graph,
defined as a pair $(V,\sim)$ where
$V$ is the {\em vertex-set} and $\sim$ an irreflexive symmetric {\em adjacency relation} on $V$.

However, in Section~\ref{ssec:trunc} it will be convenient to allow graphs possessing parallel edges;
details will be explained there.
Finally, in Section~\ref{ssec:maps}, when considering embeddings of graphs onto surfaces, we will intuitively think of a graph in a topological context as a $1$-dimensional CW complex. See that section for details.

\section{An upper bound on the signature}
\label{sec:bound}

This section is devoted to the proof of Theorems~\ref{the:1},~\ref{the:2} and~\ref{the:3}
that give an upper bound on
the number of girth cycles through an edge in a girth-regular graph and in some cases
 characterise the graphs attaining this bound.

\subsection{Moore graphs and generalised $n$-gons}

We begin by a well-known result that sets a lower bound on the number of vertices
for a $k$-regular graph of finite girth $g$.

\begin{proposition}
\label{prop:vertbound}
(Tutte~\cite[8.39]{t66}, cf.~Brouwer, Cohen \& Neumaier~\cite[\S 6.7]{bcn89})
Let $\G$ be a $k$-regular graph with $n$ vertices and finite girth $g \ge 2$. Let
$d = \lfloor g/2 \rfloor$.
Then
\begin{equation} \label{eqn:vertbound}
\pushQED{\qed}
n \ge \begin{cases}
1 + k \sum_{j=0}^{(g-3)/2} (k-1)^j = \frac{k(k-1)^d - 2}{k-2}& \text{if $g$ is odd}, \\
2 \sum_{j=0}^{(g-2)/2} (k-1)^j = 2 \frac{(k-1)^d-1}{k-2} & \text{if $g$ is even}.
\end{cases}
\qedhere
\popQED
\end{equation}
\end{proposition}

\begin{note}
\label{note:prop:vertbound}
Let $\G$ be a $k$-regular graph of girth $g$
for which equality holds in \eqref{eqn:vertbound}.
If $g$ is odd, then such an extremal graph is called a {\em Moore graph}.
It is well known (see~\cite{d73} or~\cite{bi73}, for example)
 that a Moore graph is either a complete graph, an odd cycle,
or has girth $5$ and valence $k \in \{3, 7, 57\}$.
Of the latter, the first two cases uniquely determine the Petersen graph
and the Hoffman-Singleton graph, respectively,
while no example is known for $k = 57$.

If the girth $g$ is even, then $\G$ is an incidence graph
of a generalised $(g/2)$-gon of order $(k-1, k-1)$
(see~\cite{t59} or~\cite[\S 6.5]{bcn89}, for example).
For $k = 2$, we have ordinary polygons,
and their incidence graphs are even cycles.
For $k \ge 3$, such generalised $(g/2)$-gons
only exist if $g/2 \in \{2, 3, 4, 6\}$ (see~\cite[Theorem~1]{fh64}).
In particular, if $k=3$,
then $\G$ is the incidence graph of a generalised $d$-gon of order $(2, 2)$,
where $d \in \{2, 3, 4, 6\}$.
For $d = 2$, this is a geometry with three points incident to three lines,
so its incidence graph is $K_{3,3}$.
For $d = 3$, we get the Fano plane,
and its incidence graph is the Heawood graph,
which is the unique cubic arc-transitive graph on $14$ vertices.
For $d = 4$, there is a unique generalised quadrangle of order $(2, 2)$,
cf.~Payne \& Thas~\cite[5.2.3]{pt09},
and its incidence graph is the Tutte-Coxeter graph,
also known as the Tutte $8$-cage,
which is the unique connected cubic arc-transitive graph on $30$ vertices.
For $d = 6$, there is a unique dual pair
of generalised hexagons of order $(2, 2)$,
cf.~Cohen \& Tits~\cite{ct85},
and their incidence graph (on 126 vertices), also known as the Tutte $12$-cage, is not vertex-transitive.
However, the latter graph is edge-transitive,
making it {\em semisymmetric}
-- in fact, it is the unique cubic semisymmetric graph on $126$ vertices
(see~\cite{cmmp06}, where this graph is denoted by S126).
\end{note}

\subsection{Proof of Theorems~\ref{the:1},~\ref{the:2} and~\ref{the:3}}

Equipped with these facts, we are now ready to prove Theorems~\ref{the:1},~\ref{the:2} and~\ref{the:3}.
Let us thus assume that $\G$ is a simple connected girth-regular graph of valence $k\ge 3$, let $g$ be its girth and let
$(a_1, a_2, \ldots, a_k)$ be its signature. Set $d = \lfloor g/2 \rfloor$.

In order to prove Theorem~\ref{the:1}, we need to show that
$a_k \le (k-1)^d$, or equivalently, that $\epsilon(e) \le (k-1)^d$ for every edge $e$ of $\G$.

For an integer $i$ and a vertex $v$ of $\G$, let $S_i(v)$ denote the set of vertices of $\G$ that are at distance $i$ from $v$,
and for an edge $uv$ of $\G$, let $D^i_j(u,v) = S_i(u) \cap S_j(v)$.
If $i$ and $j$ are integers such that $|i-j|\ge 2$, then clearly $D^i_j(u,v) = \emptyset$.

Now let $uv$ be an arbitrary edge of $\G$ and let $i\in \{2,\ldots,d\}$.
For simplicity, let $D_j^i=D_j^i(u,v)$. If $A$ and $B$ are two sets of vertices of $\G$, let
$E(A,B)$ be the set of edges with one end-vertex in $A$ and the other in $B$.
Since $g\ge 2d$, the following facts can be easily deduced:
\begin{enumerate}[(1)]
\item \label{f1} $D_i^i = \emptyset$ if $i \le d-1$;
\item \label{f2} each of $D^{i-1}_{i}$ and $D_{i-1}^{i}$ is an independent set;
\item \label{f3} each vertex in $D^{i-1}_{i}$
has precisely one neighbour in $D^{i-2}_{i-1}$, and if $i\le d-1$,
precisely $k-1$ neighbours in $D^{i}_{i+1}$;
\pitem \label{f3p} each vertex in $D^{i}_{i-1}$
has precisely one neighbour in $D^{i-1}_{i-2}$, and if $i\le d-1$,
precisely $k-1$ neighbours in $D^{i+1}_{i}$;
\item \label{f4} $|D_{i-1}^{i}| = |D_{i}^{i-1}| = (k-1)^{i-1}$;
\item \label{f5} if $g$ is even, then $\epsilon(uv) = |E(D_d^{d-1},D^d_{d-1})|$;
\item \label{f6} if $g$ is odd, then every vertex in $D_d^d$ has precisely one neighbour in each of the sets $D^{d-1}_{d}$, $D_{d-1}^{d}$ and
               $\epsilon(uv) = |D_d^d|$.
\end{enumerate}

Henceforth, let $uv$ be an arbitrary edge of $\G$ such that $\epsilon(uv) = a_k$ and let $D_i^j = D_i^j(u, v)$.
The structure of $\G$ with respect to the sets $D_i^j$ is then depicted in Figure~\ref{fig:cmax}.

\begin{figure}[t]
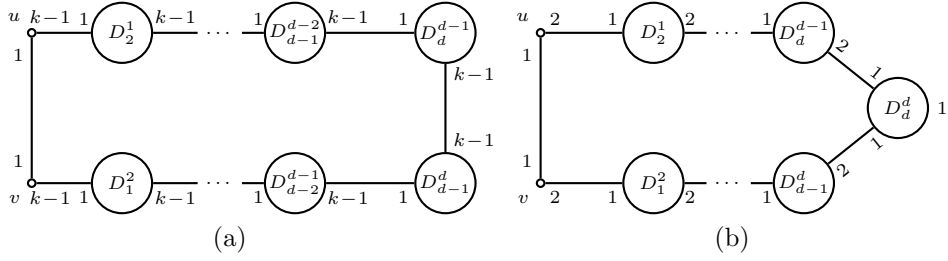

\makebox[\textwidth][c]{
\begin{tabular}{C{0.5}C{0.5}}
\leavevmode\beginpgfgraphicnamed{fig-cmaxeven}\input{tikz/cmaxeven.tikz}\endpgfgraphicnamed
&
\leavevmode\beginpgfgraphicnamed{fig-cmaxodd}\input{tikz/cmaxodd.tikz}\endpgfgraphicnamed
\\
(a) & (b)
\end{tabular}
}
\caption{The partitions of the vertices of $\G$ of girth $g$
corresponding to an edge $uv$ lying on $(k-1)^d$ girth cycles,
where $d = \lfloor g/2 \rfloor$.
(a) shows the case when $g$ is even,
while (b) shows the case when $\G$ is cubic and $g$ is odd.
The sets $D^i_j$ with $i+j < 2d$ are independent sets,
while the set $D^d_d$ in the odd case induces a perfect matching.}
\label{fig:cmax}
\end{figure}

Suppose first that $g$ is even. Let
$$
D=\bigcup_{i = 1}^d (D_{i-1}^{i} \cup D_{i}^{i-1})
$$
and observe that all of the vertices in $D$,
except possibly those in $D^{d-1}_{d}$ and $D_{d-1}^{d}$,
have all of their neighbours contained in $D$.
By \eqref{f4}, we see that
$$
|D| = 2(1+(k-1) + \ldots + (k-1)^{d-1}) = 2 \frac{(k-1)^d-1}{k-2}.
$$
Moreover, it follows from (\ref{f1}--\ref{f5}) that
$$a_k = \epsilon(uv) = |E(D_d^{d-1},D^d_{d-1})| \le (k-1) |D_d^{d-1}| = (k-1)^d.$$
 This proves Theorem~\ref{the:1} in the case when $g$ is even.
 (The case when $g$ is odd will be considered later.)

To prove Theorem~\ref{the:2}, assume that $a_k = (k-1)^d$.
Then equality holds in the above equation,
implying that $|E(D_d^{d-1},D^d_{d-1})| = (k-1) |D_d^{d-1}|$,
which means that each vertex in $D_d^{d-1}$
has $k-1$ neighbours within $D^d_{d-1}$.
This implies that every vertex from the set $D$
has all of its neighbours contained in $D$,
and by connectivity of $\G$, we see that $V(\G) = D$.
But then by Proposition~\ref{prop:vertbound}
and Note~\ref{note:prop:vertbound},
the graph $\G$ is the incidence graph
of a generalised $g/2$-gon of order $(k-1,k-1)$.
If, in addition, $k=3$ holds,
then $\Gamma$ is one of the graphs
mentioned in the statement of Theorem~\ref{the:2}.
This proves Theorem~\ref{the:2}.

Let us now move to the case where $g$ is odd, prove Theorem~\ref{the:3} and finish the proof of Theorem~\ref{the:1}.
Suppose henceforth that $g$ is odd.
Even though Theorem~\ref{the:3} is only about cubic graphs,
we will try to continue the proof without this assumption for as long as we can.
Let
$$
D=D_d^d \cup \bigcup_{i = 1}^d (D_{i-1}^{i} \cup D_{i}^{i-1})
$$
and observe that
$$
|D| \le (k-1)^d + 2(1+(k-1) + \ldots + (k-1)^{d-1}) = \frac{k(k-1)^d-2}{k-2}.
$$
If we prove that every vertex in $D$ has all of its neighbours contained in $D$, the connectivity of $\G$ will imply that $V(\G) = D$.
But then Proposition~\ref{prop:vertbound} will imply that $\G$ is
a Moore graph. Since the only cubic Moore graphs are $K_4$ and the Petersen graph,
this will then imply Theorem~\ref{the:3}.

Note that by \eqref{f2}, \eqref{f3} and \eqref{f3p},
it follows that the neighbourhoods of all vertices,
except possibly those contained in $D^{d-1}_{d}$, $D_{d-1}^{d}$ or $D_d^d$,
are contained in $D$.
By \eqref{f3}, \eqref{f3p} and \eqref{f6},
it follows that $|D_d^d| \le (k-1)|D_d^{d-1}| = (k-1)^d$,
and by \eqref{f6} we see that
$$
 a_k = \epsilon(uv) = |D_d^d| \le (k-1)^d,
$$
thus proving Theorem~\ref{the:1} also for the case when $g$ is odd.

Assume now that $a_k = (k-1)^d$. Then, by \eqref{f6}, $|D_d^d| = (k-1)^d$,
implying that every vertex in $D_d^{d-1}$ (as well as in $D^d_{d-1}$)
has $k-1$ neighbours in $D_d^d$, and thus none outside the set $D$.
To prove Theorem~\ref{the:3}, it thus suffices to show that every vertex from $D_d^d$ has all of its neighbours in $D$.

Since every vertex in $D_{d+1}^{d}$ or $D_d^{d+1}$ has to have at least one neighbour in $D_d^{d-1}$ or $D_{d-1}^{d}$, respectively, and
since all of the neighbours of vertices in the latter two sets lie in $D_d^d$, $D_{d-1}^{d-2}$ and $D_{d-2}^{d-1}$, it follows that
the sets $D_{d+1}^{d}$ and $D_d^{d+1}$ are empty.
By consequence, the sets $D_i^{i+1}(u,v)$ and $D_{i+1}^{i}(u,v)$ for $i\ge d$
are also empty.
Let us summarise that in Lemma~\ref{lem:7}.

\begin{lemma}
\label{lem:7}
Let $\G$ be a girth-regular graph of girth $2d+1$
and signature $(a_1,\ldots,a_k)$ such that $a_k = (k-1)^d$.
If $uv$ is an edge of $\G$ such that $\epsilon(uv) = a_k$,
then for $i\ge d$ the sets $D_i^{i+1}(u,v)$ and $D_{i+1}^{i}(u,v)$ are empty.
\end{lemma}

Suppose now that $V(\G) \not= D$.
Then a vertex $y\in D_d^d$ has a neighbour $w$ outside $D$.
Since the girth of $\G$ is $2d+1$,
there exists a unique path of length $d$ from $y$ to $u$.
Let $v'$ be the neighbour of $u$ through which this path passes,
and let $u'$ be a neighbour of $v'$ other than $u$
such that $\epsilon(v'u') = \epsilon(uv)$.
Let $E^i_j = D^i_j(u',v')$ and observe that by Lemma~\ref{lem:7},
the sets $E_d^{d+1}$ and $E_{d+1}^{d}$ are empty.
Furthermore, since $w$ is not in $D$ but has a neighbour $y$ in $D$,
we see that $d(w,u) = d+1$, implying that $w\in D_{d+1}^{d+1}$.

We shall now partition the set $D_d^{d-1}$ with respect to the distance to the vertices $v'$ and $u'$.
In particular, we will show that $D_d^{d-1}$ is a disjoint union of the sets
\begin{align*}
X &= D_d^{d-1} \cap E^{d-3}_{d-2}, \\
Y &= D_d^{d-1} \cap E^{d-1}_{d-2}, \\
Z &= D_d^{d-1} \cap E^{d}_{d}.
\end{align*}

To prove this, note first that a vertex in $D_d^{d-1}$
is at distance $d-1$ from $u$ and thus by \eqref{f1},
it is either at distance $d-2$ or $d$ from $v'$.
Furthermore, those vertices that are at distance $d-2$ from $v'$
are either at distance $d-3$ or $d-1$ from $u'$,
and therefore belong to $X$ or $Y$.
Now let $x$ be an element of $D_d^{d-1}$ that is at distance $d$ from $v'$. Since $E_d^{d+1} = \emptyset$, this implies
that $x$ is either in $E_d^{d-1}$ or in $E_d^{d}$. If $x\in E_d^{d-1}$, then there exist two distinct paths of length $d$ from $x$ to $v'$, one
passing through $u$ and one passing through $u'$, yielding a cycle of length at most $2d$, which is a contradiction. Hence $x\in E_d^{d}$, and therefore $x\in Z$.

We will now determine the sizes of $X$, $Y$ and $Z$. In particular, we will show that:
\begin{align*}
|X| &= (k-1)^{d-3}, \\
|Y| &= (k-2)(k-1)^{d-3}, \\
|Z| &= (k-2)(k-1)^{d-2}.
\end{align*}

To prove the first equality,
observe that $X$ consists of all the ends of paths of length $d-2$
that start with $v'u'$.
The equality for $|X|$ then follows from the fact
that there are $(k-1)^{d-3}$ such paths.
Further, note that $Y$ consists of all the ends of paths of length $d-2$
that start in $v'$ but do not pass through $u'$ or $u$.
There are $(k-2)(k-1)^{d-3}$ such paths, proving the equality for $|Y|$.
Finally, to prove the equality for $|Z|$,
observe that $Z$ consists of all the ends of paths of length $d-1$
that start in $u$ but do not pass through $v$ or $v'$;
there are clearly $(k-2)(k-1)^{d-2}$ such paths.

We will now partition the set $D_d^d$ into sets $X'$, $Y'$ and $Z'$ defined as follows. Let $x$ be a vertex of $D_d^d$ and observe that there is a unique path from $x$ to $u$ of length $d$. If this path passes through $X$,
then we let $x\in X'$, if it passes through $Y$, then we let $x\in Y'$, and if it passes through $Z$, we let $x\in Z'$.

Since each vertex in $D_d^{d-1}$ has $k-1$ neighbours in $D_d^d$ and each vertex in $D_d^d$ has precisely one neighbour in $D_d^{d-1}$, we see that
\begin{alignat*}{2}
|X'| &= (k-1)|X| &&= (k-1)^{d-2}, \\
|Y'| &= (k-1)|Y| &&= (k-2)(k-1)^{d-2}, \\
|Z'| &= (k-1)|Z| &&= (k-2)(k-1)^{d-1}.
\end{alignat*}

Observe furthermore that a vertex $x$ in $X'$, having a neighbour in $X$,
is at distance at most $d-2$ from $u'$,
but since it is at distance $d$ from $u$,
it is at distance exactly $d-2$ from $u'$.
Similarly, $d(x,v') \le d-1$ and since $d(x,u) = d$,
we see that $d(x,v') = d-1$.
In particular, $x\in E_{d-1}^{d-2}$ and thus
$$
 X' = D_d^d\cap E_{d-1}^{d-2} = E_{d-1}^{d-2}.
$$
A similar argument shows that
$$
 Y' = D_d^d\cap E_{d-1}^{d}.
$$
Let us now consider the set $Z'$,
and in particular the intersection $A = Z' \cap E_d^{d-1}$.
Note that each vertex in $Z$ must have at least one neighbour in $A$,
for otherwise it could not be at distance $d$ from $u'$.
This implies that $|A| \ge |Z| = (k-2)(k-1)^{d-2}$.
On the other hand, for a similar reason,
each vertex in $A$ must have a neighbour in $X'$.
By comparing the sizes of $A$ and $X'$,
we may thus conclude that every vertex in $X'$ has $k-2$ neighbours in $A$
and each vertex in $A$ has precisely one neighbour in $X'$.
In particular, every vertex in $X'$ has all of its neighbours in $D$,
and consequently, the vertex $w$ has no neighbours in $X'$.
Therefore, we have $y \in Y'$.
Now recall that $w\in D_{d+1}^{d+1}$, implying that $d(w,v') \ge d$.
On the other hand, $w$ has a neighbour in $Y'$,
which is a subset of $E_{d-1}^d$, implying that $d(w,v') = d$.
Since $E_{d}^{d+1} = \emptyset$, it follows that $w\in E_d^d$,
and hence there exists a path $wz_1z_2 \ldots z_{d-1}u'$
of length $d$ from $w$ to $u'$.
By considering possibilities for such a path,
one can now easily see that $z_1 \in Z'$ and $z_2\in X'$.
But then $z_1$ has at least four neighbours:
$z_2$, $w$, a neighbour in $Z$, and a neighbour in $D^d_{d-1}$,
see Figure~\ref{fig:codd}.
This contradicts our assumption that the valence $k$ is $3$.
This contradiction shows that $V(\Gamma) = D$,
and thus completes the proof of Theorem~\ref{the:3}.

\begin{figure}[t]
\makebox[\textwidth][c]{
\leavevmode\beginpgfgraphicnamed{fig-cmaxodd-proof}\input{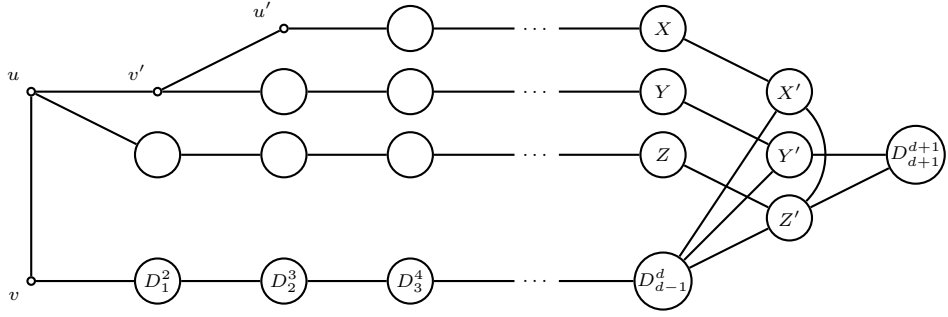}\endpgfgraphicnamed
}
\caption{The partitions of the vertices of $\G$ of girth $g$,
where $g$ is odd,
corresponding to the edges $uv$ and $u'v'$,
both lying on $2^d$ girth cycles, where $d = \lfloor g/2 \rfloor$.
Assuming there is a vertex $w \in D^{d+1}_{d+1}$,
we show that $w \in E^d_d$ has a neighbour in $Z'$,
which in turn must have at least four neighbours.}
\label{fig:codd}
\end{figure}

\section{Cubic girth-regular graphs} \label{sec:cgr}

Let us now turn our attention to cubic girth-regular graphs.
After proving a few auxiliary lemmas,
we will characterise cubic girth-regular graphs
of some specific signatures.
As an application of our analysis,
we provide a characterisation of all cubic girth-regular graphs
of girth at most $5$ in Sections~\ref{sec:g1234} and~\ref{sec:g5}.

\subsection{Auxiliary results}
\label{sec:cvt}

\begin{lemma} \label{lem:eventriineq}
If $(a, b, c)$ is the signature of a cubic girth-regular graph $\G$ of girth $g$, then:
\begin{enumerate}
\item \label{lem:eventriineq:1} $a+b+c$ is even,
\item \label{lem:eventriineq:2} $a+b \ge c$, and
\item \label{lem:eventriineq:3} if $a \ge 1$ and $c = a+b$, then $g$ is even.
\end{enumerate}
\end{lemma}

\begin{proof}
Let $u$ be a vertex of $\G$ and let
and $e_1, e_2$ and $e_3$ be the three edges
incident to $u$, lying on $a, b$ and $c$ $g$-cycles, respectively.
Further, let $x, y, z$ be the number of $g$-cycles
the $2$-paths $e_1e_2, e_2 e_3$ and $e_3e_1$ lie on,
respectively.
Clearly, we have $a = x+z$, $b = x+y$ and $c = y+z$.
Then $a+b+c = 2(x+y+z)$, showing that this sum is even.

Further we may express $x = (a+b-c)/2$, $y = (-a+b+c)/2$ and $z = (a-b+c)/2$.
Since these numbers are nonnegative, it follows that $a + b \ge c$.

Now suppose that $a \ge 1$ and $c = a+b$.
Let us call an edge $e$ with $\epsilon(e) = c$ {\em saturated} and others {\em unsaturated}.
Note that $c>b$, implying that $e_1$ and $e_2$ are unsaturated while $e_3$ is saturated.
Since $y+z = c = a+b = 2x+y+z$, we see that $x=0$. Since $u$ was an arbitrary vertex of $\G$,
this shows that a $2$-path in $\G$ consisting of two unsaturated edges belongs to no $g$-cycles.
In particular, when traversing a $g$-cycle in $\G$, saturated and unsaturated edges must alternate,
implying that $g$ is even.
\end{proof}

\begin{lemma}
\label{lem:azero}
If the signature of a cubic girth-regular graph is $(0, b, c)$, then $b = c = 1$.
\end{lemma}

\begin{proof}
Let $\G$ be a cubic girth-regular graph with signature $(0, b, c)$ and let $g$ be its girth.
By part \eqref{lem:eventriineq:2} of Lemma~\ref{lem:eventriineq}, it follows that $b = c$.
Suppose that $b > 1$.
Let $e$ be an edge of $\G$ lying on $b$ $g$-cycles,
and let $C, C'$ be two distinct $g$-cycles containing $e$.
Since $C\not = C'$, there exists a vertex $u$
such that one of the edges incident to $u$ lies on both $C$ and $C'$,
while each of the remaining two edges incident to $u$ belongs to exactly one of $C$ and $C'$.
However, this contradicts $a = 0$.
\end{proof}

\begin{corollary} \label{cor:oddg}
If $\G$ is a cubic girth-regular graph with signature $(a, b, c)$
and girth $g$, where $g$ is odd,
then $a \ne 1$.
\end{corollary}

\begin{proof}
Suppose that $a=1$. By part \eqref{lem:eventriineq:2} of Lemma~\ref{lem:eventriineq}, $c = b$ or $c = b+1$.
If $b=c$, then $a+b+c$ is odd, contradicting part \eqref{lem:eventriineq:1} of Lemma~\ref{lem:eventriineq}.
Hence $c=b+1=a+b$, and by part \eqref{lem:eventriineq:3} of Lemma~\ref{lem:eventriineq}, $g$ is even, contradicting our assumptions.
\end{proof}

\begin{lemma} \label{lem:ab}
Let $\G$ be a cubic girth-regular graph of girth $g$
with signature $(a, b, c)$.
Let $m = 2^{\lfloor g/2 \rfloor - 1}$.
Then $a \ge c - m$ and $b \le a - c + 2m$.
\end{lemma}

\begin{figure}[t]
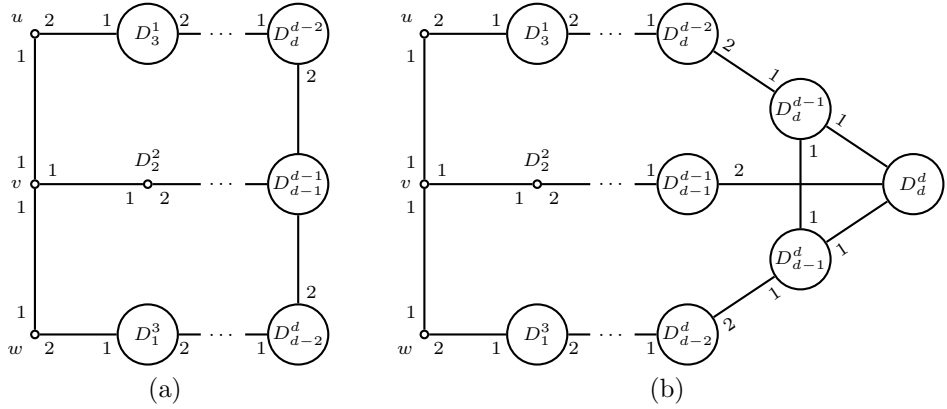

\makebox[\textwidth][c]{
\begin{tabular}{C{0.37}C{0.63}}
\leavevmode\beginpgfgraphicnamed{fig-uvweven}\input{tikz/uvweven.tikz}\endpgfgraphicnamed
&
\leavevmode\beginpgfgraphicnamed{fig-uvwodd}\input{tikz/uvwodd.tikz}\endpgfgraphicnamed
\\
(a) & (b)
\end{tabular}
}
\caption{The partitions of the vertices of $\G$ of girth $g$
corresponding to a $2$-path $uvw$ lying on $2^{d-1}$ girth cycles,
where $d = \lfloor g/2 \rfloor$.
(a) shows the case when $g$ is even,
while (b) shows the case when $g$ is odd.
The sets $D^i_j$ with $i+j < 2d$ are independent sets,
while the set $D^d_d$ may contain edges.
Note that no vertex of $D^i_i$ ($i \in \{d-1, d\}$)
with a neighbour in $D^{i-1}_{i-1}$
can have a neighbour in $D^{i-1}_i$ or $D^i_{i-1}$.}
\label{fig:uvw}
\end{figure}

\begin{proof}
Let us first show that any $2$-path in $\G$ lies on at most $m$ girth cycles.
Let $uvw$ be a $2$-path in $\G$,
and let $D^i_j$ be the set of vertices
at distance $i$ from $u$ and at distance $j$ from $w$.
Set $d = \lfloor g/2 \rfloor$.
Similarly as in the proof of Theorem~\ref{the:1},
we can see that the number of girth cycles containing the $2$-path $uvw$
equals the number of common neighbours
of vertices in the sets $D^{d-2}_d$ and $D^d_{d-2}$ if $g$ is even,
and the number of edges
between the vertices in the sets $D^{d-1}_d$ and $D^d_{d-1}$ if $g$ is odd,
see Figure~\ref{fig:uvw}.
In the even case, $|D^{d-2}_d| = |D^d_{d-2}| = 2^{d-2}$,
and each of the vertices from $D^{d-2}_d$ or $D^d_{d-2}$
may have at most two common neighbours with vertices of the other set,
so $uvw$ can lie on at most $2^{d-1} = m$ girth cycles.
In the odd case, we have $|D^{d-1}_d|, |D^d_{d-1}| \le 2^{d-1}$,
and each vertex from $D^{d-1}_d$ or $D^d_{d-1}$
may have at most one neighbour in the other set,
as otherwise we would have a cycle of length $2d < g$.
Therefore, $uvw$ can lie on at most $m$ girth cycles also in this case.

As each of $a, b, c$ is the sum of the number of girth cycles
two distinct $2$-paths sharing the central vertex lie on,
the quantity $c-a$ equals
the difference between the numbers of girth cycles two such $2$-paths lie on,
and is therefore at most $m$,
from which $a \ge c - m$ follows.
Also, the quantity $-a+b+c$
equals twice the number of girth cycles a $2$-path in $\G$ lies on,
and is therefore at most $2m$.
From this, $b \le a - c + 2m$ follows.
\end{proof}

\subsection{Dihedral schemes, truncations and signature $(0,1,1)$}
\label{ssec:trunc}

In this section we will allow graphs to have parallel edges and loops.
A graph {\em with parallel edges and loops}
is defined as a triple $(V,E,\partial)$
where $V$ and $E$ are the {\em vertex-set}
and the {\em edge-set} of the graph
and $\partial \colon E \to \{ X : X \subseteq V, |X| \le 2\}$
is a mapping that maps an edge to the set of its end-vertices.
If $|\partial(e)| = 1$, then $e$ is a {\em loop}.
Further, we let each edge consist of two mutually inverse {\em arcs},
each of the two arcs having one of the end-vertices as its {\em tail}.
If the graph has no loops, we may identify an arc with tail $v$ underlying edge $e$ with the pair $(v,e)$.
The set of arcs of a graph $\G$ is denoted by $A(\G)$ and the set of the arcs with their tail being a specific vertex $u$ by $\out_\G(u)$. The valence of a vertex $u$ is defined as
the cardinality of $\out_\G(u)$.

A {\em dihedral scheme} on a graph $\G$ (possibly with parallel edges
and loops)
is an irreflexive symmetric relation $\lr$ on the arc-set $A(\G)$ such that
the simple graph $(A(\G),\lr)$ is a $2$-regular graph each of whose connected components is the set $\out_\G(u)$ for some $u\in V(\G)$.
(Intuitively, we may think of a dihedral scheme as a collection of circles drawn around each vertex $u$ of $\G$ intersecting each of the arcs in $\out_\G(u)$
once.) Note that, according to this definition, the minimum valence of a graph admitting a dihedral scheme is at least $3$.

The group of all automorphisms of $\G$ that preserve the relation $\lr$ will be denoted by $\Aut(\G,\lr)$ and the dihedral scheme
$\lr$ is said to be {\em arc-transitive} if $\Aut(\G,\lr)$ acts transitively on $A(\G)$.

Given a dihedral scheme $\lr$ on a graph $\G$, let $\Tr(\G,\lr)$ be the simple graph whose vertices are the arcs of $\G$
and two arcs $s,t\in \G$ are adjacent in $\G$ if either $t \lr s$ or $t$ and $s$ are inverse to each other.
 The graph $\Tr(\G,\lr)$ is then
called the {\em truncation of $\G$ with respect to the dihedral scheme $\lr$}.
Note that $\Tr(\G, \lr)$ is a cubic graph
which is connected whenever $\G$ is connected.

 As we shall see in Section~\ref{ssec:maps},
 a natural source of arc-transitive dihedral schemes are arc-transitive maps (either orientable or non-orientable).
However, not all dihedral schemes arise in this way.

Clearly, the automorphism group $\Aut(\G,\lr)$ acts naturally
as a group of automorphisms of $\Tr(\G,\lr)$,
implying that $\Tr(\G,\lr)$ is vertex-transitive
whenever the dihedral scheme $\lr$ is arc-transitive.
The following result gives a characterisation
of arc-transitive dihedral schemes in group theoretical terms.
Here, the symbol $\D_d$ denotes the dihedral group of order $2d$
acting naturally on $d$ points,
while $\Z_d$ is the cyclic group acting transitively on $d$ points.

\begin{lemma}
Let $\G$ be an arc-transitive graph (possibly with parallel edges) of valence $d$ for some $d\ge 3$.
Then $\G$ admits an arc-transitive dihedral scheme
if and only if there exists an arc-transitive subgroup $G\le \Aut(\G)$ such that the group $G_u^{\out_\G(u)}$
induced by the action of the vertex stabiliser $G_u$ on the set $\out_\G(u)$ is permutation isomorphic
to the transitive action of $\D_d$, $\Z_d$ or (when $d$ is even) $\D_{\frac{d}{2}}$ on $d$ vertices.
\end{lemma}

\begin{proof}
Suppose that $\lr$ is a dihedral scheme on $\G$ and that $G=\Aut(\G,\lr)$.
Then $G_u^{\out_\G(u)}$ preserves the restriction $\lr_u$ of the
relation $\lr$ onto $\out_\G(u)$, and thus acts
as a vertex-transitive group of automorphisms on the simple graph $(\out_\G(u),\lr_u)$. Since the latter graph is a cycle of length $d$, we thus see
that $G_u^{\out_\G(u)}$ is a transitive subgroup of $\D_d$ and thus permutation
isomorphic to one of the transitive actions mentioned in the statement of the lemma.

Conversely, suppose that for some vertex $u$,
the group $G_u^{\out_\G(u)}$ is permutation isomorphic
to the transitive action of $\D_d$, $\Z_d$,
or (if $d$ is even) $\D_{\frac{d}{2}}$ on $d$ vertices.
In all three cases,
we may choose an adjacency relation $\lr_u$ on $\out_\G(u)$
preserved by $G_u^{\out_\G(u)}$ in such a way
that $(\out_\G(u),\lr_u)$ is a cycle.
For every $v\in V(\G)$,
choose an element $g_v\in G$ such that $v^{g_v} = u$,
and let $\lr_v$ be the relation on ${\out_\G(v)}$
defined by $s \lr_v t$ if and only if $s^{g_v} \lr_u t^{g_v}$.
Then clearly $(\out_\G(v),\lr_v)$ is a cycle,
implying that the union $\lr$ of all $\lr_u$ for $u\in V(\G)$
is a dihedral scheme.
Moreover, it is a matter of straightforward computation
to show that $\lr$ is invariant under $G$.
\end{proof}

We are now ready to prove the following characterisation of cubic girth-regular graphs of signature $(0,1,1)$.

\begin{theorem} \label{thm:azero}
If $\G$ is a simple cubic girth-regular graph of girth $g$ with signature $(0, 1, 1)$,
then $\G \cong \Tr(\Lambda,\lr)$, where $\lr$ is a dihedral scheme on a
$g$-regular graph $\Lambda$ (possibly with parallel edges).
Moreover, if $\G$ is vertex-transitive, then the dihedral scheme is arc-transitive.
\end{theorem}

\begin{proof}
Let $V$ be the vertex-set of $\G$,
let $\T$ be the set of girth cycles in $\G$,
let $\M$ be the set of edges that belong to no girth cycle in $\G$,
and let $G=\Aut(\G)$.
Note that since the signature of $\G$ is $(0,1,1)$, each vertex $v\in V$ is incident to exactly one
edge in $\M$ and to exactly one girth cycle in $\T$.

For an edge $v'v \in \M$,
let $C$ and $C'$ be the girth cycles that pass through $v$ and $v'$, respectively,
and let $\partial(v'v) = \{C,C'\}$.
This allows us to define a graph $\Lambda=(\T,\M, \partial)$.

Note that since $C, C' \in V(\Lambda)$ are girth cycles of $\G$,
we have $C\not = C'$, and so $\Lambda$ has no loops.
This allows us to view an arc of $\Lambda$ as
a pair $(C,e)$ where $e \in \M$ and $C$ is a girth cycle of $\G$ passing
through one of the two end-vertices of $e$. For two such pairs $(C_1,e_1)$ and $(C_2,e_2)$
we write $(C_1,e_1) \lr (C_2,e_2)$ if and only if $C_1 = C_2$ and the end-vertices of
$e_1$ and $e_2$ that belong to $C_1$ are two consecutive vertices of $C_1$.
Then $\lr$ is a dihedral scheme on $\Lambda$.
Let $\G' = \Tr(\Lambda,\lr)$.

We will now show that $\G' \cong \G$. By the definition of truncation, the vertex-set of $\G'$
equals the arc-set of $\Lambda$. For an arc $(C,e)$ of $\Lambda$ let $\varphi(C,e)$
be the unique end-vertex of $e$ that belongs to $C$.
Since each vertex of $\Gamma$ is incident to exactly one edge in $\M$ and exactly one cycle in $\T$,
it follows that $\varphi$ is a bijection between $V(\G')$ and $V(\G)$.
If $(C_1,e_1)$ and $(C_2,e_2)$ are adjacent in $\G'$, then either
$(C_1,e_1) \lr (C_2,e_2)$ or $(C_1,e_1)$ and $(C_2,e_2)$ are inverse arcs in $\G'$.
In the first case, $C_1=C_2$ and the vertices $\varphi(C_1,e_1)$ and $\varphi(C_2,e_2)$ are
adjacent on $C_1$. In the second case, $e_1 = e_2$ and
the vertices $\varphi(C_1,e_1)$ and $\varphi(C_2,e_2)$ are the two end-vertices of $e_1$. In both
cases $\varphi(C_1,e_1)$ and $\varphi(C_2,e_2)$ are adjacent in $\G$. By a similar argument
we see that whenever $\varphi(C_1,e_1)$ and $\varphi(C_2,e_2)$ are adjacent in $\G$,
$(C_1,e_1)$ and $(C_2,e_2)$ are adjacent in $\G'$.
Since both $\G$ and $\G'$ are simple graphs (one by assumption, the other by definition),
this shows that $\varphi$ is a graph isomorphism.

Suppose now that $G$ is transitive on the vertices of $\G$.
Since both sets $\T$ and $\M$ are invariant under the action of $G$,
there exists a natural action of $G$ on $\Lambda$
that preserves the dihedral scheme $\lr$;
that is, $G\le \Aut(\Lambda,\lr)$.
Now let $(C_1,e_1)$ and $(C_2,e_2)$ be two arcs of $\Lambda$,
and for $i\in\{1,2\}$,
let $v_i$ be the unique end-vertex of $e_i$ that lies on $C_i$.
Since $G$ is vertex-transitive on $\Gamma$,
there exists $g\in G$ mapping $v_1$ to $v_2$.
Since $C_i$ is the unique girth-cycle through $v_i$ for $i\in\{1,2\}$,
it follows that $C_1^g = C_2$.
Similarly,
since $e_i$ is the unique edge in $\M$ incident with $v_i$ for $i\in\{1,2\}$,
it follows that $e_1^g = e_2$.
This shows that $G$ acts transitively on the arcs of $\Lambda$.
\end{proof}

\begin{note}
Parallel edges occur in the graph $\Lambda$ as in Theorem~\ref{thm:azero}
whenever there exist two girth cycles in $\Gamma$
such that there are at least two edges
with an end-vertex in each of the two girth cycles.
In fact, it can be easily seen that in a girth-regular graph $\Gamma$
with signature $(0, 1, 1)$,
there are at most two such edges between any two girth cycles,
leading to at most two parallel edges between each two vertices,
with the exception of the case when $\Gamma$ is the $3$-prism
(see Section~\ref{sec:g1234})
and $\Lambda$ is the graph with two vertices
and three parallel edges between them.
\end{note}

\begin{note}
No nontrivial bound on the girth of the graph $\Lambda$
as in Theorem~\ref{thm:azero} can be given.
In fact, we can construct a family of graphs of constant girth
such that their truncations with respect to appropriate dihedral schemes
are cubic girth-regular graphs with signature $(0, 1, 1)$
and unbounded girth.
Let $\Lambda$ be a graph obtained by doubling all edges
in a $k$-regular graph of girth at least ${k+1 \over 2}$
-- the girth of $\Lambda$ is then $2$.
Equip $\Lambda$ with a dihedral scheme $\lr$
such that each two arcs with a common tail
belonging to two parallel edges are antipodal
in the connected component of the graph defined by $\lr$ they belong to.
Then $\Tr(\Lambda, \lr)$ is a cubic girth-regular graph of girth $k$
and signature $(0, 1, 1)$.
\end{note}

\subsection{Maps and signatures $(2,2,2)$ and $(1,1,2)$}
\label{ssec:maps}

In this section, it will be convenient to think of a graph (possibly with parallel edges)
as a topological space
having the structure of a regular 1-dimensional CW complex with the vertices of the
graph corresponding to the $0$-cells of the complex and the edges corresponding to
the $1$-cells.
A simple closed walk (that is, a closed walk that traverses each edge at most once) in the graph then corresponds to a closed curve in the corresponding topological space
which may intersect itself only in the points that correspond to the vertices of the graph.

Given a graph $\Gamma$ (viewed as a CW complex) and a set of
simple closed walks
$\T$ in $\G$,
one can construct a 2-dimensional CW complex in the following way.
First, take a collection $\mD$ of topological
disks, one for each
walk in $\T$. Then choose a
surjective continuous mapping
from the boundary of each disk to the closed curve in $\G$ representing the corresponding
walk in $\T$,
such that the preimage of each point that is not a vertex of the graph is a singleton.
Finally, identify each point of the boundary of the disk with its image under that continuous mapping.
Note that the resulting topological space is independent of the choice of the homeomorphisms $\mD$ and thus depends only on the choice of the graph and
the set of
closed walks $\T$.

When $\Gamma$ is connected and the resulting topological space is a closed surface (either orientable or non-orientable), the CW complex is also called a {\em map}. Its open $2$-cells are then called the {\em faces} of the map, the closed walks in $\T$ are called the
{\em face-cycles} and the graph $\Gamma$ is the {\em skeleton} of the map.
A map whose skeleton is a $k$-regular graph and all of whose face cycles are of length $m$ is called an $\{m,k\}$-map. The following lemma provides a sufficient condition on the set of cycles $\T$
under which the resulting $2$-dimensional CW complex is indeed a map.

\begin{lemma}
\label{lem:map}
Let $\G$ be a graph and $\T$ a set of
simple closed walks in $\G$ such that every edge of $\G$ belongs to precisely
two
walks in $\T$. For two arcs $s$ and $t$ with a common tail,
write $s\lr t$ if and only if the underlying edges of $s$ and $t$ are two consecutive edges on a
walk in $\T$.
If $\lr$ is a dihedral scheme, then $\G$ is the skeleton of a map whose face cycles are precisely the
walks in $\T$.
\end{lemma}

\begin{proof}
Let us think of $\G$ as a 1-dimensional CW complex
and let us turn it into a $2$-dimensional CW complex
by adding to it one $2$-cell for each walk in $\T$ as described above.

Let us now prove that the resulting topological space $\M$
is a closed surface.
It is clear that the internal vertices of the $2$-cells
have a regular neighbourhood.
Further, since each edge of $\G$ lies on precisely two walks in $\T$,
the internal points of edges also have a regular neighbourhood,
made up from two half-disks,
each contained in the $2$-cell glued to one of the
walks in $\T$ passing through that edge.
Finally, let $u$ be a vertex of $\G$,
let $k$ be the valence of $u$,
and let $\{s_i : i \in \Z_k\}$ be the set of arcs with the initial vertex $u$
such that $s_0\lr s_1 \lr \ldots \lr s_{k-1} \lr s_0$.
By the definition of $\lr$,
each pair of arcs $(s_i, s_{i+1})$ ($i\in \Z_k$)
lies on a unique walk $C_i$ in $\T$.
Note that $C_i \not = C_{i+1}$,
for otherwise the edge underlying $s_{i+1}$
would lie on only one walk in $\T$.
This implies that a regular neighbourhood of $u$ in $\M$
can be built by taking appropriate half-disks
from the $2$-cells corresponding to the cycles $C_i$ ($i\in \Z_k$),
and gluing them together in the order suggested by the relation $\lr$.
This shows that $\M$ is a $2$-manifold without a boundary.
Finally, since $\G$ is finite, $\M$ is compact, and thus a closed surface.
Hence, $\M$ is a map with $\G$ as its skeleton.
\end{proof}

Each face of a map can be decomposed further into {\em flags}, that is,
triangles with one vertex in the centre of a face,
one vertex in the centre of an edge on the boundary of that face
and one in a vertex incident with that edge.
In most cases, a flag can be viewed as a triple consisting of a vertex,
an edge incident to that vertex,
and a face incident to both the vertex and the edge.

An automorphism of a map is then defined as a permutation of the flags
induced by a homeomorphism of the surface that preserves the embedded graph.
A map is said to be {\em vertex-transitive} or {\em arc-transitive}
provided that its automorphism group
induces a vertex-transitive or arc-transitive group
on the skeleton of the map, respectively.

\begin{note}
\label{note:map}
If a map is built from a graph $\G$ and a set of
simple closed walks
$\T$ as in Lemma~\ref{lem:map},
then each automorphism of $\G$ that preserves the set of
walks $\T$ clearly
extends to an automorphism of the map.
\end{note}

If $\M$ is a map on a surface $\s$, then the sets $V$, $E$ and $F$ of the vertices,
edges and faces, respectively, satisfy the {\em Euler formula}
$$
|V| - |E| + |F| = \chi(\s)
$$
where $\chi(\s)$ is the {\em Euler characteristic} of the surface $\s$. It is well known that
$\chi(\s) \le 2$ with equality holding if and only if $\s$ is homeomorphic to a sphere.
Moreover, if $\chi(\s)$ is odd, then $\s$ is non-orientable.

As the following two results show, skeletons of maps arise naturally when analysing cubic vertex-transitive graphs of signature $(2,2,2)$ or $(1,1,2)$.

\begin{theorem} \label{thm:abc2}
Let $\G$ be a simple connected cubic girth-regular graph of girth $g$ and order $n$
with signature $(2, 2, 2)$.
Then $g$ divides $3n$ and $\G$ is the skeleton of a $\{g, 3\}$-map
embedded on a surface with Euler characteristic
$$\chi = n \left(\frac{3}{g} - \frac{1}{2} \right).$$
Moreover, every automorphism of $\G$ extends to an automorphism of the map.
In particular, if $\G$ is vertex-transitive, so is the map.
\end{theorem}

\begin{proof}
Let $\T$ be the set of girth cycles of $\G$.
Since the valence of $\G$ is $3$,
it follows easily that the relation $\lr$ from Lemma~\ref{lem:map}
satisfies the conditions stated in the lemma;
that is, $\lr$ is a dihedral scheme.
Lemma~\ref{lem:map} thus yields a map $\M$
whose skeleton is $\G$ and whose face-cycles are precisely the walks in $\T$;
in particular, $\M$ is a $\{g,3\}$-map, as claimed.

Since $\G$ is a cubic graph with $n$ vertices, it has $3n/2$ edges, and since each vertex lies on three face-cycles and since each face-cycle contains $g$ vertices, the map $\M$ has $3n/g$ faces (showing that $g$ must divide $3n$). The Euler characteristic of $\M$ thus equals
$n - \frac{3n}{2} + \frac{3n}{g} = n(\frac{3}{g} - \frac{1}{2}).$

Since every automorphism of $\G$ preserves $\T$,
it extends to an automorphism of $\M$ (see Note~\ref{note:map}).
\end{proof}

Theorem~\ref{thm:abc2} has the following interesting consequence.

\begin{corollary}
There exists only finitely many connected cubic girth-regular graphs with signature $(2,2,2)$ of girth at most $5$.
\end{corollary}

\begin{proof}
Suppose that $\G$ is a connected cubic girth-regular graph with signature $(2,2,2)$ of girth $g$
and order $n$.
By Theorem~\ref{thm:abc2}, $\G$ is a skeleton of a map on a surface of Euler characteristic $\chi = n (3/g - 1/2)$. Hence, if $g\le 5$, then $\chi \ge n/10$, and since $\chi\le 2$, it follows that $n\le 20$.
\end{proof}

\begin{note}
For each $g \ge 6$,
there are infinitely many girth-regular graphs of girth $g$
with signature $(2, 2, 2)$.
\end{note}

If $\M$ is a map and $\G$ is its skeleton, then one can define a dihedral scheme $\lr$ on
$\G$ by letting $s \lr t$ whenever the arcs $s$ and $t$ have a common tail and
the underlying edges of $s$ and $t$ are two consecutive edges on some face-cycle of $\M$.
The truncation $\Tr(\G,\lr)$ is then simply referred to as the {\em truncation of the map} $\M$ and denoted $\Tr(\M)$.
Note that this construction in some sense complements Lemma~\ref{lem:map}. We are now equipped for a characterisation of cubic girth-regular graphs with signature $(1,1,2)$.

\begin{theorem}
\label{thm:trieq1}
Let $\G$ be a simple connected cubic girth-regular graph of girth $g$
with $n$ vertices and signature $(1, 1, 2)$.
Then $g$ is even and $\G$ is the truncation of some map $\M$
with face cycles of length $g/2$.
In particular, $g/2$ divides $n$.
Moreover, if $\G$ is vertex-transitive,
$\M$ is an arc-transitive $\{g/2, \ell\}$-map for some $\ell > g$.
\end{theorem}

\begin{proof}
By part \eqref{lem:eventriineq:3} of Lemma~\ref{lem:eventriineq}
we know that $g$ is even and in particular, $g\ge 4$.
Let $\X$ be the set of edges of $\G$ that belong to exactly one girth cycle
and let $\Y$ be the set of edges that belong to two girth cycles.
Since the signature of $\G$ is $(1,1,2)$,
every vertex of $\G$ is incident to two edges in $\X$ and one edge in $\Y$.
Consequently, the edges in $\Y$ form a perfect matching of $\G$
and the subgraph induced by the edges in $\X$
is a union of vertex-disjoint cycles of $\G$
that cover all the vertices of $\G$.
Let us denote the set of these cycles by $\C$.

Observe also that two edges in $\X$ sharing a common end-vertex, say $v$,
cannot be two consecutive edges on the same girth cycle,
for otherwise that would be a unique girth cycle through $v$,
contradicting the fact that the third edge incident with $v$
belongs to two girth cycles.
Since the edges in $\Y$ form a complete matching of $\G$,
the same holds for the edges in $\Y$,
implying that the edges on any girth cycle
alternate between the sets $\X$ and $\Y$.

For an edge $e$ in $\Y$ with end-vertices $u$ and $v$,
let $C_u$ and $C_v$ be the unique cycles in $\C$
that pass through $u$ and $v$, respectively,
and define $\partial(e)$ to be the pair $\{C_u,C_v\}$.
Let $\Lambda = (\C,\Y,\partial)$.
Note that since the edges of $\Lambda$
are precisely those edges of $\G$ that belong to $\Y$,
we may think of the arc-set $A(\Lambda)$
as being the set of arcs of $\G$ that underlie edges in $\Y$.
Note also that it may happen that for some $e \in \Y$,
we may have $C_u = C_v$ and then the graph $\Lambda$ has loops.
If $D$ is a girth cycle of $\G$,
then the edges of $D$ that belong to $\Y$
induce a simple closed walk in the graph $\Lambda$ of length $g/2$,
which we denote $\hat{D}$.

Let $\T$ be the set of walks $\hat{D}$
where $D$ runs through the set of girth cycles of $\G$.
Since edges of $\Lambda$ correspond to the edges of $\G$
that pass through two girth cycles of $\G$,
each edge of $\Lambda$ belongs to two walks in $\T$.
As $|\Y| = n/2$, it follows that $g/2$ divides $n$.

Let $\lr$ be the relation on the arcs of $\Lambda$
defined by $\T$ as explained in Lemma~\ref{lem:map}.
It is easy to see that $\lr$ is a dihedral scheme.
Indeed, let $C\in \C$ be a vertex of $\Lambda$ viewed as a cycle in $\G$
and let $v_0,v_1,\ldots, v_{k-1} \in V(\G)$ be its vertices
listed in a cyclical order as they appear on $C$.
Further, for each $i\in \Z_k$,
let $s_{i}$ be the arc of $\G$ with tail $v_i$
that underlies an edge contained in $\Y$.
The arc $s_i$ can thus also be viewed as an arc of $\Lambda$.
Observe that $\out_\Lambda(C) = \{ s_0, s_1, \ldots, s_{k-1}\}$
and that $s_0\lr s_1\lr \ldots \lr s_{k-1} \lr s_0$.
In particular, $\lr$ is a dihedral scheme.

By Lemma~\ref{lem:map},
there exists a map $\M$ with skeleton $\Lambda$
in which $\T$ is the set of face-cycles.
Moreover, $\lr$ equals the dihedral scheme arising from that map.

Let $\G' = \Tr(\M)$ and let $s$ be a vertex of $\G'$.
Then $s$ is an arc of $\Lambda$ and thus also an arc of $\G$
underlying an edge in $\Y$.
By letting $\varphi(s)$ be the tail of $s$ (viewed as an arc of $\G$),
we define a mapping $\varphi \colon V(\G') \to V(\G)$.
Note that the mapping which assigns to a vertex $v\in V(\G)$
the unique arc of $\G$ with tail $v$ that underlies an edge in $\Y$
is the inverse of $\varphi$,
showing that $\varphi$ is a bijection.
Furthermore,
note that two vertices $s$ and $t$ of $\G'$ are adjacent in $\G'$
if and only if one of the following happens:
(1) they are inverse to each other as arcs of $\Lambda$; or
(2) they have a common tail and $s\lr t$.
In case (1), $\varphi(s)$ and $\varphi(t)$ are adjacent in $\G$
via an edge in $\Y$,
while in case (2),
$\varphi(s)$ and $\varphi(t)$ are adjacent in $\G$ via an edge in $\X$.
Conversely, if for some $s,t\in V(\G')$,
the images $\varphi(s)$ and $\varphi(t)$ are adjacent in $\G$,
then either $s$ and $t$ are inverse to each other as arcs of $\Lambda$
(this happens if $\varphi(s)$ and $\varphi(t)$ form an edge in $\Y$),
or $s$ and $t$ have a common tail and $s\lr t$
(this happens if $\varphi(s)$ and $\varphi(t)$ form an edge in $\X$).
In both cases, $s$ and $t$ are adjacent in $\G'$.
This implies that $\varphi$ is an isomorphism of graphs
and thus $\G\cong \Tr(\M)$, as claimed.

Since every automorphism of $\G$
preserves each of the sets $\Y$ and $\X$
(and thus also $\C$),
it clearly induces an automorphism of the graph $\Lambda$
which preserves the set $\T$. In particular, every automorphism of $\G$
induces an automorphism of the map $\M$.

Finally, suppose that $\G$ is vertex-transitive.
Then all cycles of $\C$ have the same length $\ell > g$.
As each vertex of a cycle of $\C$ is incident to precisely one edge of $\Y$,
it follows that $\Lambda$ is an $\ell$-regular graph,
and $\M$ is then a $\{g/2, \ell\}$-map.
Let $G$ be a group of automorphisms of $\G$
acting transitively on $V(\G)$.
Note that every vertex of $\G$
is the tail of precisely one arc of $\G$ that underlies an edge of $\Y$.
In view of our identification of the arcs of $\G'$
with the arcs of $\G$ that underlie an edge in $\Y$,
we thus see that the transitivity of the action of $G$ on $V(\G)$
implies the transitivity of the action of $G$ on the arcs of $\M$.
\end{proof}

\section{Cubic girth-regular graphs of girths $3$ and $4$}
\label{sec:g1234}

Before stating the theorem about girth-regular cubic graphs of girth $3$,
let us point out that every cubic graph admits a unique dihedral scheme,
which is preserved by every automorphism of the graph.
This allows us to talk about truncations of cubic graphs without specifying the dihedral scheme.

\begin{theorem} \label{thm:g3}
Let $\G$ be a connected cubic girth-regular graph of girth $3$.
Then one of the following holds:
\begin{enumerate}[(a)]
\item $\Gamma$ is isomorphic to the complete graph $K_4$;
\item $\G$ has signature $(0,1,1)$ and is isomorphic to the truncation of a cubic graph.
\end{enumerate}
\end{theorem}

\begin{proof}

Let $(a,b,c)$ be the signature of $\Gamma$.
By Theorem~\ref{the:1} it follows that $c\le 2$.
If $c=2$, then
Theorem~\ref{the:3} implies that $\G$ is isomorphic to $K_4$.
On the other hand, if $c=1$, then Lemmas~\ref{lem:eventriineq} and~\ref{lem:azero} imply that
the signature of $\Gamma$ is $(0, 1, 1)$, and
 by Theorem~\ref{thm:azero}, it follows that $\Gamma$ is the truncation of a cubic graph.
\end{proof}

Let us now move our attention to graphs of girth $4$.
Before stating the classification theorem,
let us define two families of cubic vertex-transitive graphs.

For $n\ge 3$, let the {\em $n$-Möbius ladder} $M_n$
be the Cayley graph $\Cay(\Z_{2n}, \{-1, 1, n\})$.
Note that such a graph has girth $4$.
The graph $M_n$ has signature $(4,4,4)$ if $n=3$
(in this case it is isomorphic to the complete bipartite graph $K_{3,3}$),
and $(1,1,2)$ if $n\ge 4$.
An $n$-Möbius ladder can also be seen as the skeleton of the truncation
of the $\{2, 2n\}$-map with a single vertex embedded on a projective plane.

For $n \ge 3$, the {\em $n$-prism} $Y_n$ is defined as the
Cartesian product $C_n \Box K_2$ or, alternatively,
as the Cayley graph $\Cay(\Z_n \times \Z_2, \{(-1, 0), (1, 0),$ $(0, 1)\})$.
The girth of $Y_3$ is $3$, while the girth of $Y_n$ for $n\ge 4$ is $4$.
The graph $Y_n$ has signature $(2,2,2)$ if $n=4$
(in this case it is isomorphic to the cube $Q_3$),
and $(1,1,2)$ if $n\ge 5$.
An $n$-prism can also be seen as the skeleton of the truncation
of the $\{2, n\}$-map with two vertices embedded on a sphere,
i.e., an $n$-gonal hosohedron.

\begin{theorem} \label{thm:g4}
Let $\G$ be a connected cubic girth-regular graph of girth $4$.
Then $\G$ is isomorphic to one of the following graphs:
\begin{enumerate}[(a)]
\item the $n$-Möbius ladder $M_n$ for some $n\ge 3$;
\item the $n$-prism $Y_n$ for some $n\ge 4$;
\item $\Tr(\Lambda,\lr)$ for some tetravalent graph $\Lambda$ and a dihedral scheme $\lr$ on $\Lambda$.
\end{enumerate}
\end{theorem}

\begin{figure}[t]
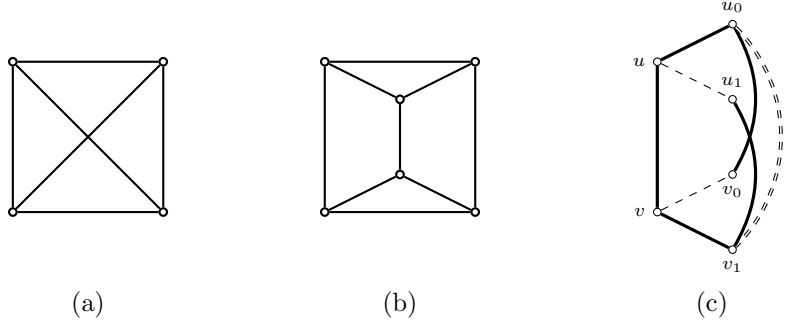

\makebox[\textwidth][c]{
\begin{tabular}{C{0.3}C{0.3}C{0.3}}
\leavevmode\beginpgfgraphicnamed{fig-y2}\input{tikz/y2.tikz}\endpgfgraphicnamed
&
\leavevmode\beginpgfgraphicnamed{fig-a3}\input{tikz/a3.tikz}\endpgfgraphicnamed
&
\leavevmode\beginpgfgraphicnamed{fig-g4-c3}\input{tikz/g4-c3.tikz}\endpgfgraphicnamed
\\
(a) & (b) & (c)
\end{tabular}
}
\caption{(a) The graph $K_4$ of girth $3$ with signature $(2, 2, 2)$.
(b) The graph $Y_3$ of girth $3$ with signature $(1, 1, 2)$.
(c) Constructing a graph of girth $4$ with $c = 3$.
The dashed edges should lie on two $4$-cycles,
however the doubled edge already lies on three $4$-cycles.}
\label{fig:g34}
\end{figure}

\begin{proof}
Let $(a,b,c)$ be the signature of $\Gamma$.
By Theorem~\ref{the:1}, we see that $c\le 4$, and by Theorem~\ref{the:2},
if $c=4$, then the signature of $\Gamma$ is $(4,4,4)$
and $\Gamma \cong K_{3,3} \cong M_3$.

Suppose now that $c=3$.
Then, by Lemma~\ref{lem:eventriineq}, $a+b$ is odd,
and by Lemma~\ref{lem:azero}, $a\ge 1$.
Hence either $a=1$ and then $b=2$, or $a=2$ and then $b=3$.
The possible signatures in this case are thus $(1,2,3)$ and $(2,3,3)$.
Let us show that neither can occur.

Let $uv$ be an edge of $\G$ lying on three $4$-cycles,
and $u_0, u_1$ and $v_0, v_1$ be the remaining neighbours of $u$ and $v$,
respectively.
There must be three edges with one end-vertex in $\{u_0, u_1\}$
and the other in $\{v_0,v_1\}$;
without loss of generality,
these edges are $u_0 v_0$, $u_0 v_1$ and $u_1 v_1$
(see Figure~\ref{fig:g34}(c)).
Then the edges $uu_0$ and $vv_1$ already lie on three $4$-cycles,
so we have $b = 3$, and thus $a=2$.
In particular, $\epsilon(uu_1)=\epsilon(vv_0) = 2$.
Then the edge $u_1v_1$, being incident to both $u_1$ and $v_1$,
belongs to precisely three $4$-cycles, that is $\epsilon(u_1v_1) = 3$.
Similarly, $\epsilon(u_0v_0) = 3$.
It follows that the edge $u_0 v_1$ lies on precisely two $4$-cycles.
However, we have already determined three $4$-cycles on which $u_0v_1$ lies;
these are $uu_0v_1v$, $v_0u_0v_1v$, and $uu_0v_1u_1$.
This contradiction shows that the case $c = 3$ is not possible.

Suppose now that $c=2$. By Lemma~\ref{lem:eventriineq}, $a+b$ is even, and by Lemma~\ref{lem:azero}, $a\ge 1$. Hence the
signature of $\Gamma$ is either $(1,1,2)$ or $(2,2,2)$.

If $(a, b, c) = (1, 1, 2)$,
then, by Theorem~\ref{thm:trieq1},
$\G$ is the skeleton of the truncation
of a connected map $\M$ with face cycles of length $2$.
Since every edge belongs to two faces
and every face is surrounded by two edges,
the number of faces equals the number of edges.
The Euler characteristics $\chi(\s)$ of the underlying surface $\s$
thus equals $|V(\M)|$.
Since $\chi(\s) \le 2$, it follows that $\M$ has one or two vertices,
depending on whether $\s$ is the projective plane or the sphere
-- in particular,
the skeleton of $\M$ is an $\ell$-regular graph for some $\ell > 4$.
If $\M$ has one vertex only, then it consists of $\ell/2$
loops embedded onto the projective plane
in such a way that its truncation
is the Möbius ladder $M_n$ with $n = \ell/2 \ge 4$,
see Figure~\ref{fig:proj}(a)
(note that $M_3\cong K_{3,3}$ has signature $(4, 4, 4)$).
On the other hand, if $\M$ has two vertices,
then $\M$ is the map with two vertices and $\ell$ parallel edges
embedded onto the sphere.
The graph $\G$ is then isomorphic
to the $n$-prism $Y_n$ with $n = \ell \ge 5$.

If $(a, b, c) = (2, 2, 2)$,
then, by Theorem~\ref{thm:abc2},
$\G$ is the skeleton of a $\{4, 3\}$-map
embedded on a surface of Euler characteristic $\chi = n/4 > 0$.
As above, $\chi\le 2$ and thus $\chi = 1$ or $2$.
For $\chi = 1$, we get the hemicube on the projective plane
(see Figure~\ref{fig:proj}(b)),
and its skeleton is isomorphic to $K_4$ of girth $3$.
For $\chi = 2$, we get the cube on a sphere,
and its skeleton is isomorphic to $Y_4$ with signature $(2, 2, 2)$.
This completes the case $c=2$.

If $c=1$, then since $a+b+c$ is even, we see that
$a=0$ and $b=1$, and then by Theorem~\ref{thm:azero},
$\G$ is the truncation of a $4$-regular graph
with respect to some dihedral scheme.
\end{proof}

\begin{figure}[t]
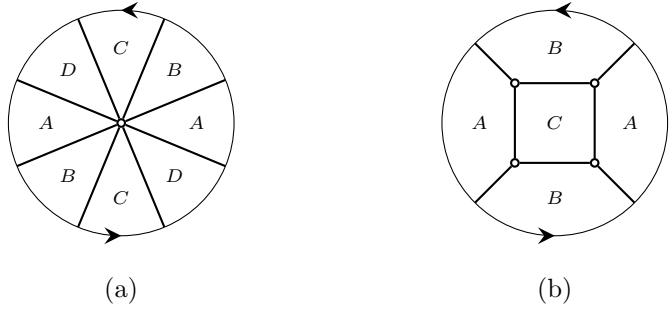

\makebox[\textwidth][c]{
\begin{tabular}{C{0.4}cC{0.4}}
\leavevmode\beginpgfgraphicnamed{fig-projective}\input{tikz/projective.tikz}\endpgfgraphicnamed
&\quad&
\leavevmode\beginpgfgraphicnamed{fig-hemicube}\input{tikz/hemicube.tikz}\endpgfgraphicnamed
\\
(a) && (b)
\end{tabular}
}
\caption{(a) A $\{2, 8\}$-map with a single vertex,
four edges and four labelled faces
embedded on the projective plane.
Its truncation has the graph $M_4$ with signature $(1, 1, 2)$
as its skeleton.
(b) The hemicube on the projective plane with labelled faces.
Its skeleton is the graph $K_4$.}
\label{fig:proj}
\end{figure}

\section{Cubic girth-regular graphs of girth $5$}
\label{sec:g5}

\begin{theorem} \label{thm:g5}
Let $\G$ be a connected cubic girth-regular graph of girth $5$.
Then either the signature of $\G$ is $(0, 1, 1)$ and $\G$
is the truncation of a $5$-regular graph
with respect to some dihedral scheme, or $\G$ is isomorphic to the Petersen graph or to the dodecahedron graph.
\end{theorem}

\begin{proof}
Let $(a,b,c)$ be the signature of $\Gamma$.
By Theorem~\ref{the:1}, we see that $c\le 4$.
If $c=4$, then by Theorem~\ref{the:2} and Theorem~\ref{the:3}, the signature of $\Gamma$ is $(4,4,4)$ and
 $\Gamma$ is isomorphic to the Petersen graph. We may thus assume that $c\le 3$.

If $a=0$, then by Lemma~\ref{lem:azero} the signature of $\Gamma$ is $(0,1,1)$, and then by Theorem~\ref{thm:azero},
$\Gamma$ is the truncation of a $5$-regular graph with respect to some dihedral scheme. Moreover, by Corollary~\ref{cor:oddg}, $a\not =1$.
We may thus assume that $a\ge 2$.

If $c=2$,
then the signature of $\Gamma$ is $(2,2,2)$ and by Theorem~\ref{thm:abc2},
$\G$ is the skeleton of a $\{5, 3\}$-map
embedded on a surface of Euler characteristic $\chi = n/10$,
where $n$ is the order of the graph $\G$.
In particular, $\chi \in \{1,2\}$.
If $\chi = 1$, then $n=10$ and since the girth of $\G$ is $5$,
Proposition~\ref{prop:vertbound} and Note~\ref{note:prop:vertbound}
imply that $\G$ is the Petersen graph (whose signature is in fact $(4,4,4)$).
If $\chi = 2$,
then $n=20$ and $\G$ is the skeleton of a $\{5,3\}$-map on the sphere.
It is well known that there is only one such map, namely the dodecahedron.

Finally, suppose that $c=3$.
Then, by Lemma~\ref{lem:eventriineq}, $a+b$ is odd,
and since $a\ge 2$, the signature of $\G$ is $(2,3,3)$.
We will now show that this possibility does not occur.

Let $uv$ be an edge of $\G$ lying on three $5$-cycles,
and $u_0, u_1, v_0, v_1$ be vertices of $\G$ with adjacencies
$u_0 \sim u \sim u_1$ and $v_0 \sim v \sim v_1$.
Then there should be three vertices
adjacent to one of $u_0, u_1$ and one of $v_0, v_1$.
Without loss of generality,
let $w_{00}, w_{10}, w_{11}$ be vertices such that
$u_0 \sim w_{00} \sim v_0 \sim w_{10} \sim u_1 \sim w_{11} \sim v_1$.
Further, let $x$ be the neighbour of $u_0$ other than $u$ and $w_{00}$,
and let $y$ be the neighbour of $v_1$ other than $v$ and $w_{11}$,
see Figure~\ref{fig:g5c3}(a).
Observe that $x\not = y$,
for otherwise the edge $uv$ would belong to four $5$-cycles.
Note also that $x$ is not adjacent to any of the three neighbours of $v$,
for otherwise the girth of $\G$ would be at most $4$.

The signature implies that for each vertex,
two edges incident to it lie on three $5$-cycles.
Suppose that $uu_0$ lies on three $5$-cycles.
As $x$ and $v$ have no common neighbours,
$w_{00}$ and $x$ must have a common neighbour with $u_1$,
so we have $w_{00} \sim w_{11}$ and $x \sim w_{10}$,
see Figure~\ref{fig:g5c3}(b).
But then the edge $uu_1$ lies on four $5$-cycles, contradiction.
Therefore, the edge $uu_0$ must lie on two $5$-cycles,
and a similar argument shows the same for $vv_1$.
Thus, the arcs $uu_1$, $vv_0$, $u_0 x$, $u_0 w_{00}$,
$v_1 w_{11}$ and $v_1 y$ must lie on three $5$-cycles,
see Figure~\ref{fig:g5c3}(c).

\begin{figure}[t]
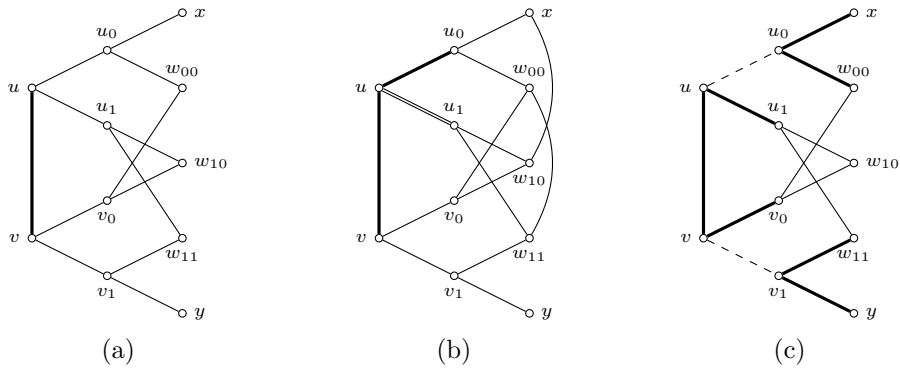

\makebox[\textwidth][c]{
\begin{tabular}{C{0.3}C{0.35}C{0.3}}
\leavevmode\beginpgfgraphicnamed{fig-g5-c3}\input{tikz/g5-c3.tikz}\endpgfgraphicnamed
&
\leavevmode\beginpgfgraphicnamed{fig-g5-c3a}\input{tikz/g5-c3a.tikz}\endpgfgraphicnamed
&
\leavevmode\beginpgfgraphicnamed{fig-g5-c3b}\input{tikz/g5-c3b.tikz}\endpgfgraphicnamed
\\
(a) & (b) & (c)
\end{tabular}
}
\caption{Constructing a graph of girth $5$ with $c = 3$.
The bold edges lie on three $5$-cycles,
and the dashed edges lie on two $5$-cycles.
The general setting is shown in (a).
In (b), the arc $(u, u_0)$ is assumed to lie on three $5$-cycles,
but the doubled edge then lies in four $5$-cycles.
In (c), the obtained distribution of edges among cycles is shown,
which, however, cannot be completed.}
\label{fig:g5c3}
\end{figure}

Since the edge $u_0 w_{00}$ lies on three $5$-cycles,
there should be three vertices adjacent to one of $u, x$
and one of $v_0$ and the remaining neighbour of $w_{00}$.
Similarly, $v_1 w_{11}$ lying on three $5$-cycles
implies that there should be three vertices adjacent to one of $v, y$
and one of $u_1$ and the remaining neighbour of $w_{11}$.
As $w_{10}$ is the only potential common neighbour
for $x, v_0$ and for $y, u_1$,
it follows that at least one of these pairs does not have a common neighbour.
Without loss of generality,
we may assume that $x$ and $v_0$ do not have a common neighbour.
The vertex $u$ already has a common neighbour with $v_0$,
and it must also have a common neighbour
with the remaining neighbour of $w_{00}$.
Then the remaining neighbour of $w_{00}$ must be $w_{11}$,
which however has no common neighbour with $x$, contradiction.
Therefore, $(a, b, c) = (2, 3, 3)$ is not possible.
\end{proof}

\section{Concluding remarks}

Theorem~\ref{the:main} gives a complete classification of
simple connected cubic girth-regular graphs of girths up to $5$.
While extending the classification to non-simple graphs
(i.e., girths $1$ and $2$) is straightforward,
increasing the girth leads to exponentially many more possible signatures.
For example, the census of connected cubic vertex-transitive graphs
on at most $1280$ vertices by Potočnik, Spiga and Verret~\cite{psv13}
shows that $9$ distinct signatures appear among graphs of girth $6$,
while many more signatures are allowed by the results
in Sections~\ref{sec:intro},~\ref{sec:bound} and Subsection~\ref{sec:cvt}.
A classification of connected cubic vertex-transitive graphs of girth $6$
will thus be given in a follow-up paper.

\begin{footnotesize}
\bibliographystyle{abbrv}
\bibliography{reference}

\begin{thebibliography}{10}

\bibitem{bi73}
E.~Bannai and T.~Ito.
\newblock On finite {M}oore graphs.
\newblock {\em J. Fac. Sci. Univ. Tokyo Sect. IA Math.}, 20:191--208, 1973.

\bibitem{bcn89}
A.~E. Brouwer, A.~M. Cohen, and A.~Neumaier.
\newblock {\em Distance-regular graphs}, volume~18 of {\em Ergebnisse der
  Mathematik und ihrer Grenzgebiete (3) [Results in Mathematics and Related
  Areas (3)]}.
\newblock Springer-Verlag, Berlin, 1989.
\newblock \doi{10.1007/978-3-642-74341-2}.

\bibitem{ct85}
A.~M. Cohen and J.~Tits.
\newblock On generalized hexagons and a near octagon whose lines have three
  points.
\newblock {\em European J. Combin.}, 6(1):13--27, 1985.
\newblock \doi{10.1016/S0195-6698(85)80017-2}.

\bibitem{cmmp06}
M.~Conder, A.~Malnič, D.~Marušič, and P.~Potočnik.
\newblock A census of semisymmetric cubic graphs on up to $768$ vertices.
\newblock {\em J. Algebraic Combin.}, 23(3):255--294, 2006.
\newblock \doi{10.1007/s10801-006-7397-3}.

\bibitem{cn07}
M.~Conder and R.~Nedela.
\newblock Symmetric cubic graphs of small girth.
\newblock {\em J. Combin. Theory Ser. B}, 97(5):757--768, 2007.
\newblock \doi{10.1016/j.jctb.2007.01.001}.

\bibitem{cz17}
M.~Conder and S.~S. Zemljič, 2017.
\newblock Private communication with M.~Conder.

\bibitem{d73}
R.~M. Damerell.
\newblock On {M}oore graphs.
\newblock {\em Proc. Cambridge Philos. Soc.}, 74:227--236, 1973.

\bibitem{ejs19}
E.~Eiben, R.~Jajcay, and P.~Šparl.
\newblock Symmetry properties of generalized graph truncations.
\newblock {\em J. Combin. Theory Ser. B}, 137:291--315, 2019.
\newblock \doi{10.1016/j.jctb.2019.01.002}.

\bibitem{fh64}
W.~Feit and G.~Higman.
\newblock The nonexistence of certain generalized polygons.
\newblock {\em J. Algebra}, 1:114--131, 1964.
\newblock \doi{10.1016/0021-8693(64)90028-6}.

\bibitem{jkm18}
R.~Jajcay, G.~Kiss, and {\v{S}}.~Miklavič.
\newblock Edge-girth-regular graphs.
\newblock {\em European J. Combin.}, 72:70--82, 2018.
\newblock \doi{10.1016/j.ejc.2018.04.006}.

\bibitem{km09}
K.~Kutnar and D.~Marušič.
\newblock A complete classification of cubic symmetric graphs of girth 6.
\newblock {\em J. Combin. Theory Ser. B}, 99(1):162--184, 2009.
\newblock \doi{10.1016/j.jctb.2008.06.001}.

\bibitem{pt09}
S.~E. Payne and J.~A. Thas.
\newblock {\em Finite generalized quadrangles}.
\newblock EMS Series of Lectures in Mathematics. European Mathematical Society
  (EMS), Z\"urich, second edition, 2009.
\newblock \doi{10.4171/066}.

\bibitem{psv13}
P.~Potočnik, P.~Spiga, and G.~Verret.
\newblock Cubic vertex-transitive graphs on up to $1280$ vertices.
\newblock {\em J. Symbolic Comput.}, 50:465--477, 2013.
\newblock \doi{10.1016/j.jsc.2012.09.002}.

\bibitem{pw07}
P.~Potočnik and S.~Wilson.
\newblock Tetravalent edge-transitive graphs of girth at most $4$.
\newblock {\em J. Combin. Theory Ser. B}, 97(2):217--236, 2007.
\newblock \doi{10.1016/j.jctb.2006.03.007}.

\bibitem{t59}
J.~Tits.
\newblock Sur la trialit\'{e} et certains groupes qui s'en d\'{e}duisent.
\newblock {\em Inst. Hautes Études Sci. Publ. Math.}, (2):13--60, 1959.
\newblock \url{http://www.numdam.org/item?id=PMIHES_1959__2__13_0}.

\bibitem{t66}
W.~T. Tutte.
\newblock {\em Connectivity in graphs}.
\newblock Number~15 in Mathematical Expositions. University of Toronto Press,
  Toronto, 1966.

\end{thebibliography}
\end{footnotesize}

\end{document}